\documentclass[reqno,12pt]{amsart}

\usepackage[utf8]{inputenc}
\usepackage{microtype}
\usepackage{amsfonts}
\usepackage{amsmath}
\usepackage{graphicx}
\usepackage{hyperref}
\usepackage{color}

\usepackage{amssymb}
\usepackage{enumitem}

\usepackage{tikz}
\usetikzlibrary{topaths}

\usepackage{a4wide}

\usepackage{empheq}

\newtheorem{theorem}{Theorem}[section]
\newtheorem*{theorem*}{Theorem}

\newtheorem{lemma}[theorem]{Lemma}

\newtheorem{proposition}[theorem]{Proposition}
\newtheorem*{proposition*}{Proposition}

\newtheorem{corollary}[theorem]{Corollary}
\newtheorem*{corollary*}{Corollary}

\newtheorem{cit}[theorem]{Citation}
\newtheorem*{conjecture*}{Conjecture}

\newtheorem*{question*}{Question}

\newtheorem*{main:BNSR_BF}{Theorem~\ref{thrm:BNSR_BF}}
\newtheorem*{main:cat0}{Theorem~\ref{thrm:X_CAT0}, Remark~\ref{rmk:Vbr_CAT0}}
\newtheorem*{main:hyp}{Theorem~\ref{thrm:hyp}}

\theoremstyle{definition}
\newtheorem{definition}[theorem]{Definition}
\newtheorem{remark}[theorem]{Remark}

\newcommand{\Z}{\mathbb{Z}}
\newcommand{\N}{\mathbb{N}}

\newcommand{\R}{\mathbb{R}}

\newcommand{\arcs}{\mathcal{A}}
\newcommand{\matcharc}{\mathcal{MA}}

\newcommand{\lk}{\operatorname{lk}}

\newcommand{\dlk}{\lk^\downarrow\!}
\newcommand{\alk}{\lk^\uparrow\!}

\newcommand{\defeq}{\mathbin{\vcentcolon =}}

\newcommand{\ty}{\texttt}

\DeclareMathOperator{\Hom}{Hom}

\DeclareMathOperator{\Symm}{Symm}
\DeclareMathOperator{\F}{F}
\DeclareMathOperator{\CAT}{CAT}
\DeclareMathOperator{\Stab}{Stab}

\usepackage{ifthen}
\usepackage{xargs}
\newcommand{\optionalarg}[2]{
\ifthenelse{\equal{#2}{}}{%
#1}{%
#1(#2)}
}

\newcommand{\Fbr}%
   {F_{\operatorname{br}}}                 

\newcommand{\Vbr}%
   {V_{\operatorname{br}}}                 
   
\newcommand{\Pbr}%
   {P_{\operatorname{br}}}                 

\numberwithin{equation}{section}

\begin{document}

\title{Geometric structures related to the braided Thompson groups}
\date{\today}


\author{Matthew C.~B.~Zaremsky}
\address{Department of Mathematics and Statistics, University at Albany (SUNY), Albany, NY 12222}
\email{mzaremsky@albany.edu}

\begin{abstract}
In previous work, joint with Bux, Fluch, Marschler and Witzel, we proved that the braided Thompson groups are of type $\F_\infty$. The proof utilized certain contractible cube complexes, which in this paper we prove are $\CAT(0)$. We then use this fact to compute the geometric invariants $\Sigma^m(\Fbr)$ of the pure braided Thompson group $\Fbr$. Only the first invariant $\Sigma^1(\Fbr)$ was previously known. A consequence of our computation is that as soon as a subgroup of $\Fbr$ containing the commutator subgroup $[\Fbr,\Fbr]$ is finitely presented, it is automatically of type $\F_\infty$.
\end{abstract}

\maketitle
\thispagestyle{empty}

\section*{Introduction}

A \emph{classifying space} for a group $G$ is a CW complex $X$ with $\pi_1(X)\cong G$ and $\pi_k(X)=0$ for all $k\ne 1$. We say a group is of \emph{type $F_n$} if it admits a classifying space with compact $n$-skeleton, and \emph{type $F_\infty$} if it is of type $\F_n$ for all $n$. In \cite{bux16} it was proved that the braided Thompson groups $\Vbr$ and $\Fbr$ are of type $\F_\infty$. The proof involved establishing contractible cube complexes for the groups to act on nicely.

In this paper we return to these cube complexes, and prove that they are $\CAT(0)$. We then use this to compute the higher Bieri--Neumann--Strebel--Renz (BNSR) invariants of $\Fbr$. The BNSR-invariants of a group $G$ are a sequence of geometric invariants
\[
\Sigma^1(G)\supseteq \Sigma^2(G)\supseteq \Sigma^3(G)\supseteq \cdots
\]
with $\Sigma^m(G)$ defined whenever $G$ is of type $\F_m$, which encode a vast amount of information about certain subgroups of $G$. In particular they precisely determine which subgroups of $G$ containing $[G,G]$ are of type $\F_m$. The first invariant, $\Sigma^1(G)$, is known as the Bieri--Neumann--Strebel (BNS) invariant. The BNS-invariant was developed by Bieri, Neumann and Strebel in \cite{bieri87}, and the higher invariants by Bieri and Renz in \cite{bieri88}. In general the BNSR-invariants of a group are extremely difficult to compute, and for groups where the problem is relevant a complete picture exists only in very few cases. For example, for the pure braid groups $PB_n$, Koban, McCammond and Meier fully computed the BNS-invariant $\Sigma^1(PB_n)$ for all $n$ \cite{koban15}, but even a full computation of $\Sigma^2(PB_4)$ remains open \cite{zaremsky17sepPB}.

In \cite{zaremsky18} the BNS-invariant $\Sigma^1(\Fbr)$ of the pure braided Thompson group $\Fbr$ was computed. It consists of a $3$-sphere with two points removed. In this paper we compute all the higher BNSR-invariants $\Sigma^m(\Fbr)$, which turn out to coincide for all $m\ge 2$. They consist of the $3$-sphere minus the convex hull of the two points missing from $\Sigma^1(\Fbr)$. This shows that in terms of BNSR-invariants, $\Fbr$ exhibits behavior similar to the generalized Thompson groups $F_{n,\infty}$ (see \cite{zaremsky17}).

\medskip

Our main result is the following (with notation explained in Section~\ref{sec:bnsr_BF}).

\begin{main:BNSR_BF}
For $m\ge 2$, the BNSR-invariant $\Sigma^m(\Fbr)$ for $\Fbr$ consists of all points on the sphere $\Sigma(\Fbr)=S^3$ except for the convex hull of $[\phi_{\ty{0}}]$ and $[\phi_{\ty{1}}]$.
\end{main:BNSR_BF}

A consequence of this (Corollary~\ref{cor:fp_Finfty}) is that every finitely presented subgroup of $\Fbr$ containing the commutator subgroup $[\Fbr,\Fbr]$ is of type $\F_\infty$.

A key step in proving Theorem~\ref{thrm:BNSR_BF} (more precisely, it leads to the crucial Corollary~\ref{cor:Fbrn_Finfty}) is the following result, which is of independent interest:

\begin{main:cat0}
The Stein--Farley cube complexes for $\Vbr$ and $\Fbr$ are $\CAT(0)$.
\end{main:cat0}

The property of a metric space being $\CAT(0)$ is an important notion of non-positive curvature that is fundamental in geometric group theory, and this result connects the braided Thompson groups to this world.

This paper is organized as follows. We recall some background in Section~\ref{sec:background} and discuss our groups of interest in Section~\ref{sec:gps}. In Sections~\ref{sec:cubes} and~\ref{sec:arc} we consider the Stein--Farley cube complexes and prove Theorem~\ref{thrm:X_CAT0}. Finally in Section~\ref{sec:bnsr_BF} we compute the $\Sigma^m(\Fbr)$.

\subsection*{Acknowledgments} Thanks are due to Javier Aramayona and Rodrigo de Pool for catching a mistake in a previous version of this paper.

\section{Background}\label{sec:background}

In this section we collect some background material on $\CAT(0)$ cube complexes, ascending HNN-extensions, BNSR-invariants and Morse theory.

\subsection{$\CAT(0)$ cube complexes}\label{sec:geom}

A \emph{cube complex} is an affine cell complex in which the cells are cubes, such that the intersection of two cubes is either empty or is a face of each (this is called ``cubical complex'' in \cite[Definition~I.7.32]{bridson99}). Note that links in a cube complex are simplicial complexes. For the technical definition of \emph{$CAT(0)$ space} see, e.g., \cite[Definition~II.1.1]{bridson99}. We will only be concerned with $\CAT(0)$ cube complexes, for which we have the following characterization that we can take to be a definition:

\begin{definition}[$\CAT(0)$ cube complex]\label{def:cat0}
A simply connected cube complex $X$ is called \emph{$CAT(0)$} if the link of every $0$-cube is a flag complex.
\end{definition}

In the finite dimensional case the equivalence of this definition to the technical one is due to Gromov (see \cite{gromov87} and \cite[Theorem~II.5.20]{bridson99}), and the assumption of finite dimensionality is alleviated by \cite[Theorem~B.8]{leary13}.

It is a classical fact that $\CAT(0)$ spaces are contractible \cite[Corollary~II.1.5]{bridson99}.

\begin{definition}[Combinatorially convex]\cite[Definition~B.2]{leary13}
A subcomplex $Y$ of a $\CAT(0)$ cube complex $X$ is called \emph{combinatorially convex} if it is connected and every link in $Y$ of a $0$-cube $y$ in $Y$ is a full subcomplex of the link of $y$ in $X$.
\end{definition}

The following is well known, for example see \cite[Lemma~2.12]{haglund08}.

\begin{cit}\label{cit:loc_cvx}
If $Y$ is a combinatorially convex subcomplex of a $\CAT(0)$ cube complex $X$, then $Y$ is itself $\CAT(0)$, hence contractible.
\end{cit}

In particular, to show that a connected subcomplex of a $\CAT(0)$ cube complex is contractible it suffices to check that links in the subcomplex are full, which is a very straightforward local condition.

\subsection{Ascending HNN-extensions}\label{sec:hnn}

The classical definition of an ascending HNN-extension is as follows.

\begin{definition}[External ascending HNN-extension]
Let $B$ be a group and $\phi\colon B\to B$ an injective endomorphism. The \emph{(external) ascending HNN-extension} of $B$ with respect to $\phi$ is the group
\[
B*_{\phi,t}\defeq \langle B,t\mid b^t=\phi(b) \text{ for all } b\in B \rangle\text{,}
\]
where $b^t$ denotes $t^{-1}bt$. If $\phi$ is not surjective we say the HNN-extension is \emph{strictly ascending}. We call $B$ the \emph{base group} and $t$ the \emph{stable element} of $B*_{\phi,t}$.
\end{definition}

Since we actually care about realizing existing groups as being isomorphic to ascending HNN-extensions, the notion of an ``internal'' ascending HNN-extension is most relevant for us.

\begin{cit}[Internal ascending HNN-extension]\cite[Lemma~3.1]{geoghegan01}\label{cit:int_hnn}
Let $G$ be a group, $B$ a subgroup and $z$ an element of $G$. Suppose $B$ and $z$ generate $G$ and that $B^z\subseteq B$. Assume that $z^n\in B$ if and only if $n=0$ (for instance this holds if $B^z\subsetneq B$). Then $G\cong B*_{\phi,t}$, where $\phi\colon B\to B$ is $b\mapsto b^z$.
\end{cit}

If $G$ is a group as in Citation~\ref{cit:int_hnn} we will often refer to it admitting an \emph{ascending HNN-decomposition}. In this case we write $G=B*_z$ to indicate that $G$ admits an ascending HNN-decomposition via the base group $B$ and the stable element $z$.

\subsection{BNSR-invariants}\label{sec:bnsr}

Let $G$ be a group. Consider the euclidean vector space $\Hom(G,\R)$, whose elements $\chi\colon G\to\R$ are called \emph{characters} of $G$. Declare that two characters are equivalent if they are positive scalar multiples of each other, and define $\Sigma(G)$ to be the set of equivalence classes of non-trivial characters. One can naturally view $\Sigma(G)$ as a sphere, of dimension one less than the dimension of $\Hom(G,\R)$, and we call it the \emph{character sphere} of $G$. The BNSR-invariants are certain subsets $\Sigma^1(G)\supseteq \Sigma^2(G)\supseteq\cdots$ of $\Sigma(G)$, defined as follows.

\begin{definition}[BNSR-invariants]\label{def:bnsr}
Suppose a group $G$ acts cellularly and cocompactly on an $(m-1)$-connected CW-complex $X$ such that the stabilizer of each $p$-cell is of type $\F_{m-p}$ (so in particular $G$ must be of type $\F_m$). For $0\ne \chi\in\Hom(G,\R)$ such that every stabilizer lies in $\ker(\chi)$, we can choose a map $h_\chi \colon X\to \R$ satisfying $h_\chi(g.x)=\chi(g)+h_\chi(x)$ for all $g\in G$ and $x\in X$. For $t\in\R$ let $X_{\chi\ge t}\defeq h_\chi^{-1}([t,\infty))$. The $m$th \emph{Bieri--Neumann--Strebel--Renz (BNSR) invariant} $\Sigma^m(G)$ is then
\[
\Sigma^m(G)\defeq \{[\chi]\in \Sigma(G)\mid (X_{\chi\ge t})_{t\in\R} \text{ is essentially $(m-1)$-connected}\}\text{.}
\]
\end{definition}

Here $(X_{\chi\ge t})_{t\in\R}$ is \emph{essentially $(m-1)$-connected} if for all $t\in\R$ there exists $-\infty<s\le t$ such that the inclusion of $X_{\chi\ge t}$ into $X_{\chi\ge s}$ induces the trivial map in $\pi_k$ for all $k\le m-1$. It turns out $\Sigma^m(G)$ is well defined up to the choice of $X$ and $h_\chi$ (see for example \cite[Definition~8.1]{bux04}). (Note that the hypothesis that every stabilizer lies in $\ker(\chi)$ is clearly necessary for $h_\chi$ to exist, but has occasionally been accidentally omitted in the literature, e.g., in \cite{bux04} and \cite{zaremsky17}.)

It is standard to denote $\Sigma(G)\setminus \Sigma^m(G)$ by $\Sigma^m(G)^c$, and to write $\Sigma^\infty(G)$ for the intersection of all the $\Sigma^m(G)$ over $m\in\N$.

One of the most important applications of the BNSR-invariants is the following.

\begin{cit}\cite[Theorem~1.1]{bieri10}\label{cit:fin_props}
Let $G$ be a group of type $\F_m$ and let $[G,G]\le H\le G$. Then $H$ is of type $\F_m$ if and only if for every $[\chi]\in \Sigma(G)$ such that $\chi(H)=0$, we have $[\chi]\in\Sigma^m(G)$.
\end{cit}

Definition~\ref{def:bnsr} is quite technical, but in our present work we will actually not need to worry about the details in the definition. Instead we will rely on various established results for computing the invariants, which we discuss now.

This first result will be our primary tool for proving that a character does lie in an invariant.

\begin{cit}\cite[Theorem~2.4]{meier01}\label{cit:stab_restrict}
Let $G$ be a group acting cocompactly on an $(m-1)$-connected complex $X$. Let $\chi\in\Hom(G,\R)$ be a character such that for any cell $\sigma$ in $X$ with dimension $\dim(\sigma)\le m$, the restriction $\chi|_{\Stab_G(\sigma)}$ is non-trivial. If for all $\sigma$ with $\dim(\sigma)<m$, we have that $\Stab_G(\sigma)$ is of type $\F_{m-\dim(\sigma)}$ and $[\chi|_{\Stab_G(\sigma)}]\in \Sigma^{m-\dim(\sigma)}(G_\sigma)$, then $[\chi]\in\Sigma^m(G)$.
\end{cit}

This next result (more precisely its contrapositive) will be our primary tool for proving that a character does not lie in an invariant.

\begin{cit}\cite[Corollary~2.8]{meinert97}\label{cit:hi_dim_quotients}
Let $\pi\colon G\to H$ be a split epimorphism of groups. Let $\chi$ be a character of $H$, so $\chi\circ\pi$ is a character of $G$. For any $m\ge 1$, if $[\chi\circ\pi]\in\Sigma^m(G)$ then $[\chi]\in\Sigma^m(H)$.
\end{cit}

As a remark, in the $m=1$ case the result holds even if $\pi$ does not split.

In the case that $G$ decomposes as an ascending HNN-extension we have the following two results, the second of which is a special case of Citation~\ref{cit:stab_restrict}.

\begin{cit}\cite[Theorem~2.1]{bieri10}\label{cit:hnn_kill}
Let $G$ decompose as an ascending HNN-extension $B*_z$. Let $\chi\in\Hom(G,\R)$ be the character defined by $\chi(B)=0$ and $\chi(z)=1$. If $B$ is of type $\F_m$ then $[\chi]\in \Sigma^m(G)$. If $B$ is finitely generated and $B*_z$ is strictly ascending then $[-\chi]\in \Sigma^1(G)^c$.
\end{cit}

\begin{cit}\cite[Theorem~2.3]{bieri10}\label{cit:hnn_restrict}
Let $G$ decompose as an ascending HNN-extension with base group $B$. Let $\chi\in\Hom(G,\R)$ such that $\chi|_B \ne 0$. If $B$ is of type $\F_\infty$ and $[\chi|_B]\in\Sigma^\infty(B)$ then $[\chi]\in \Sigma^\infty(G)$.
\end{cit}

Finally, the following is a useful tool for groups with well understood centers.

\begin{cit}\cite[Theorem~2.1]{meier01}\label{cit:center_survives}
Let $G$ be a group of type $\F_m$ and let $\chi\in\Hom(G,\R)$. If $\chi(Z(G))\ne 0$ then $[\chi]\in\Sigma^m(G)$.
\end{cit}

\subsection{Morse theory}\label{sec:morse}

Bestvina--Brady discrete Morse theory, developed in \cite{bestvina97}, has proven to be an indispensable tool in the study of finiteness properties of groups. Its original application done by Bestvina and Brady in \cite{bestvina97} can be viewed as a determination of precisely in which BNSR-invariants a certain character of a right-angled Artin group lies, and afterward Meier--Meinert--VanWyk \cite{meier98} and, independently, Bux--Gonzalez \cite{bux99}, gave a complete picture of all the BNSR-invariants of right-angled Artin groups. Many subsequent applications of Bestvina--Brady Morse theory, especially in the study of Thompson-like groups, use a combination of Morse theory with Brown's Criterion, first formulated in \cite{brown87}. In what follows we will need to use both ``direct'' discrete Morse theory and also the combination with Brown's Criterion, which we will state after providing some setup.

Let $X$ be an affine cell complex, as in \cite{bestvina97} (for example a cube complex). Let $h\colon X\to \R$ be a map whose restriction $h|_C \to \R$ to any cell $C$ in $X$ is affine. We call $h$ a \emph{Morse function} if it is non-constant on positive dimensional cells and $h(X)\subseteq \R$ is discrete. (Both of these conditions often need to be relaxed when dealing with BNSR-invariants, see for example \cite[Definition~1.5]{witzel18}, but it turns out we do not need to worry about this here.) The conditions ensure that for every cell $C$, the function $h$ achieves its maximum and minimum values on $C$ at unique $0$-faces. If $h$ achieves its minimum (maximum) value on $C$ at the $0$-face $x$ then we say $C$ is \emph{ascending (descending)} with respect to $x$. The \emph{ascending (descending) link} of a $0$-cell $x$, denoted $\alk x$ ($\dlk x$) is the subcomplex of the link $\lk x$ consisting of directions into those cells that are ascending (descending) with respect to $x$. If we need to consider links taken only in some subcomplex $Y$ of $X$, we will write things like $\lk_Y$, $\alk_Y$, $\dlk_Y$.

For $p\le q$ let $X^{p\le h\le q}$ denote the full subcomplex of $X$ spanned by those $0$-cells $x$ with $p\le h(x)\le q$. We allow $p=-\infty$ and write $X^{h\le q}$ in this case, and similarly allow $q=\infty$ and write $X^{p\le h}$. If we write $<$ instead of $\le$ in the notation then this is reflected in the conditions. We now state the Morse Lemma we will use; it follows from the proof of \cite[Corollary~2.6]{bestvina97} and is a special case of the more general Morse Lemma used in \cite[Lemma~1.7]{witzel18} and \cite[Lemma~1.4]{zaremsky17}.

\begin{lemma}[Morse Lemma]\label{lem:morse}
Let $h\colon X\to \R$ be a discrete Morse function on an affine cell complex, as above. Let $p,q,r\in \R\cup\{\pm\infty\}$ with $p\le q\le r$. If for every $0$-cell $x$ in $X^{q<h\le r}$ the descending link $\dlk_{X^{p\le h}}x$ of $x$ in $X^{p\le h}$ is $(n-1)$-connected, then the inclusion $X^{p\le h\le q}\to X^{p\le h\le r}$ induces isomorphisms in the homotopy groups $\pi_k$ for $k\le n-1$ and an epimorphism in $\pi_n$. If for every $0$-cell $x$ in $X^{p\le h<q}$ the ascending link $\alk_{X^{h\le r}}x$ of $x$ in $X^{h\le r}$ is $(n-1)$-connected, then the inclusion $X^{q\le h\le r}\to X^{p\le h\le r}$ induces isomorphisms in the homotopy groups $\pi_k$ for $k\le n-1$ and an epimorphism in $\pi_n$.
\end{lemma}

The combination of Morse theory plus Brown's Criterion we will use is as follows, and is immediate from Lemma~\ref{lem:morse} and \cite[Corollary~3.3]{brown87}.

\begin{lemma}\label{lem:brown}
Let $G$ be a group acting cellularly on a contractible affine cell complex $X$. Let $h\colon X\to \R$ be a $G$-invariant discrete Morse function such that the sublevel sets $X^{h\le q}$ are $G$-cocompact. Assume every cell stabilizer in $G$ is of type $\F_\infty$. Suppose that for all $n$ there exists $q$ such that for all $0$-cells $x$ with $h(x)>q$ the descending link $\dlk x$ is $(n-1)$-connected. Then $G$ is of type $\F_\infty$.
\end{lemma}

\section{The groups}\label{sec:gps}

In this section we discuss our groups of interest. First we recall some background on braid groups and pure braid groups. Then we discuss the braided Thompson group $\Vbr$ and pure braided Thompson group $\Fbr$, and a crucial family of subgroups $\Fbr(\ty{w})$ for $\ty{w}\in\{\ty{0},\ty{1}\}^*$.

\subsection{Braid groups}\label{sec:braid}

The \emph{$n$-strand braid group} $B_n$ is defined by its standard \emph{Artin presentation}
\begin{align*}
B_n \defeq \left\langle s_1,\dots,s_{n-1}\left|
\begin{array}{cc}
s_i s_j = s_j s_i &\text{ whenever } |i-j|>1 \text{ and }\\
s_i s_{i+1} s_i = s_{i+1}s_i s_{i+1} &\text{ for all } 1\le i\le n-2
\end{array}
\right.\right\rangle\text{.}
\end{align*}
An element of $B_n$ can be viewed as a picture of $n$ strands braiding around each other, with $s_i$ represented by a braid where the $i$th strand goes up and crosses under the $(i+1)$st. Here we always number strands at the bottom, so for example $s_1\in B_2$ is given by the picture in Figure~\ref{fig:cross} (as opposed to its mirror image).

\begin{figure}[htb]
 \begin{tikzpicture}[line width=0.8pt]\centering
  \draw (1,0) to[out=-90, in=90] (0,-2);
  \draw[white, line width=4pt] (0,0) to[out=-90, in=90] (1,-2);
  \draw (0,0) to[out=-90, in=90] (1,-2);
 \end{tikzpicture}
 \caption{The element $s_1$ in $B_2$.}\label{fig:cross}
\end{figure}
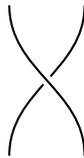

Denoting the symmetric group $\Symm(\{1,\dots,n\})$ by $S_n$, there is a natural epimorphism $\rho\colon B_n \to S_n$ given by introducing the additional relations $s_i^2=1$ for all $1\le i\le n-1$. The resulting presentation for $S_n$ is its standard \emph{Coxeter presentation}. Visually, given a braid $b$, if the $i$th strand counting from the bottom is the $j$th strand counting from the top then $\rho(b)(i)=j$. The kernel of $\rho$ is the \emph{$n$-strand pure braid group} $PB_n$. Note that when a braid is pure, the numbering of strands at the top and at the bottom is the same, so in particular for any $1\le i\le n$, deleting the $i$th strand yields a homomorphism $PB_n\to PB_{n-1}$, which is easily seen to be surjective. Removing collections of strands thus gives us epimorphisms $PB_n\to PB_m$ for $m<n$. Note that $PB_2$ is generated by $s_1^2$, so we can fix an isomorphism $PB_2\to \Z$ via $s_1^2\mapsto 1$ and get a family of maps from $PB_n$ to $\Z$. Namely, for each $1\le i<j\le n$ we have an epimorphism $\omega_{i,j}\colon PB_n \to \Z$ given by deleting all but the $i$th and $j$th strands.

\begin{definition}[Winding number]
Let $1\le i<j\le n$. The \emph{winding number} of the $i$th and $j$th strands of a pure braid $p\in PB_n$ is the number $\omega_{i,j}(p)\in\Z$.
\end{definition}

It turns out that the $\omega_{i,j}$ form a basis for $\Hom(PB_n,\R)$, so $\Hom(PB_n,\R)\cong \R^{\binom{n}{2}}$ and $\Sigma(PB_n)\cong S^{\binom{n}{2}-1}$. The BNS-invariant $\Sigma^1(PB_n)$ was computed for all $n$ by Koban--McCammond--Meier in \cite{koban15}, but very little is known about the higher BNSR-invariants (they are all defined, since (pure) braid groups are of type $\F_\infty$). One thing that is known is that the inclusions $\Sigma^m(PB_n)\subseteq \Sigma^{m-1}(PB_n)$ are proper if and only if $m\le n-2$ \cite{zaremsky17sepPB}. Fortunately, even though computing all the $\Sigma^m(\Fbr)$ will require us to know something about the $\Sigma^m(PB_n)$, it will turn out to be enough to just understand how characters of $PB_n$ interact with the center $Z(PB_n)$.

The center $Z(PB_n)$ is isomorphic to $\Z$, generated by an element $\Delta_n^2$ that is the square of
\[
\Delta_n \defeq s_1 \cdots s_{n-1} s_1 \cdots s_{n-2} \cdots s_1 s_2 s_1 \text{.}
\]
Visually, in $\Delta_n^2$ all the strands rotate rigidly around by 360 degrees. In particular $\omega_{i,j}(\Delta_n^2)=1$ for all $i,j$, and since the $\omega_{i,j}$ form a basis for $\Hom(PB_n,\R)$ this makes it easy to tell whether or not a character contains $Z(PB_n)$ in its kernel. This will be very useful later, when coupled with Citation~\ref{cit:center_survives}.

\subsection{Braided Thompson groups}\label{sec:Fbr}

The classical Thompson groups $F\subset T\subset V$, introduced in unpublished notes of Thompson in the 1960's, serve as important examples and counterexamples in group theory. In particular $T$ and $V$ were the first known examples of finitely presented infinite simple groups. The group $V$ can be viewed as a relative to Coxeter groups, for example in some sense it is built out of symmetric groups. It has a natural ``Artinification,'' called ``braided $V$,'' with a similar relationship to $V$ as the braid groups have to the symmetric groups. This group, often denoted $BV$ or $\Vbr$, was introduced independently by Brin and Dehornoy \cite{brin07,dehornoy06}, and has proven to have a variety of interesting and surprising properties. Like $V$, it is finitely presented \cite{brin06} and of type $\F_\infty$ \cite{bux16}. Further results include an inspection of metric properties of $\Vbr$ \cite{burillo09} and normal subgroups \cite{zaremsky18}. In particular, while $\Vbr$ is not simple like $V$ is, since it surjects onto $V$ much like $B_n$ surjects onto $S_n$, it is true that every proper normal subgroup of $\Vbr$ lies in the kernel of this quotient to $V$ \cite{zaremsky18}. A ``pure'' braided Thompson group, usually denoted $BF$ or $\Fbr$\footnote{The notation $BG$ often stands for a classifying space of a group $G$, so even though there is not actually any risk of confusion we will stick to writing $\Vbr$ and $\Fbr$ instead of $BV$ and $BF$.}, was introduced by Brady--Burillo--Cleary--Stein \cite{brady08}, and shown to be finitely presented. It, like $\Vbr$, is of type $\F_\infty$ \cite{bux16}.

A convenient model for elements of $\Vbr$ and $\Fbr$ is given by so called strand diagrams. An element of $\Vbr$ is represented by a triple $(T_-,b,T_+)$ for $T_\pm$ a pair of trees (by which we will always mean finite rooted binary trees) with the same number of leaves, say $n$, and $b$ an element of the $n$-strand braid group $B_n$. These data form a \emph{strand diagram} by drawing $T_-$ with its root at the top and its leaves at the bottom, drawing $T_+$ below $T_-$ with its root at the bottom and its leaves at the top, and connecting the leaves of $T_-$ to those of $T_+$ with the braid $b$. See Figure~\ref{fig:BV_element} for an example.

\begin{figure}[htb]
 \begin{tikzpicture}[line width=0.8pt]
  \draw
   (0,-2) -- (2,0) -- (4,-2)   (1.5,-1.5) -- (1,-2)   (1,-1) -- (2,-2)   (3.5,-1.5) -- (3,-2)
   (0,-2) to [out=-90, in=90] (1,-4)   (2,-2) to [out=-90, in=90] (3,-4);
  \draw[white, line width=4pt]
   (1,-2) to [out=-90, in=90] (0,-4)   (4,-2) to [out=-90, in=90] (2,-4);
  \draw
   (1,-2) to [out=-90, in=90] (0,-4)   (4,-2) to [out=-90, in=90] (2,-4);
  \draw[white, line width=4pt]
  (4,-2) to [out=-90, in=90] (2,-4);
  \draw
  (4,-2) to [out=-90, in=90] (2,-4);
  \draw[white, line width=4pt]
  (3,-2) to [out=-90, in=90] (4,-4);
  \draw
  (3,-2) to [out=-90, in=90] (4,-4);
  \draw
   (0,-4) -- (2,-6) -- (4,-4)   (2.5,-5.5) -- (2,-5) -- (3,-4)   (1,-4) -- (2,-5)   (2,-4) -- (2.5,-4.5);
 \end{tikzpicture}
 \caption{An element of $\Vbr$.}\label{fig:BV_element}
\end{figure}
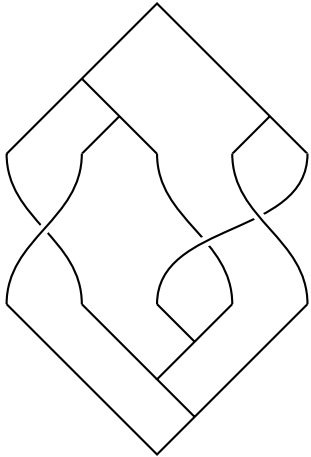

We referred above to triples $(T_-,b,T_+)$ as ``representing'' elements of $\Vbr$, and by this we mean that different such triples might yield the same element of $\Vbr$; more formally, elements of $\Vbr$ are equivalence classes of triples. Two triples are \emph{equivalent} if they can be obtained from each other by a finite sequence of expansions and reductions.

\begin{definition}[Expansion/reduction]
Let $(T_-,b,T_+)$ be a triple as above. An \emph{expansion} of $(T_-,b,T_+)$ is a triple of the form $(T_-',b',T_+')$, where $T_+'$ is obtained from $T_+$ by adding a caret to some leaf, say the $k$th, then bifurcating the $k$th strand of $b$ (counting at the bottom) into two parallel strands to get $b'$, and finally adding a caret to the $\rho(k)$th leaf of $T_-$ to get $T_-'$. We also call $(T_-,b,T_+)$ a \emph{reduction} of $(T_-',b',T_+')$.
\end{definition}

We write $[T_-,b,T_+]$ for the equivalence class of $(T_-,b,T_+)$ under this equivalence relation. Two elements $[T_-,b,T_+]$ and $[U_-,c,U_+]$ of $\Vbr$ can be multiplied in the following way. First perform expansions to get $[T_-,b,T_+]=[T_-',b',T_+']$ and $[U_-,c,U_+]=[U_-',c',U_+']$ such that $T_+'=U_-'$, and then declare
\[
[T_-,b,T_+][U_-,c,U_+]\defeq [T_-',b'c',U_+']\text{.}
\]
It is straightforward to check that this is a group operation on $\Vbr$. The identity is $[\cdot,1,\cdot]$ for $\cdot$ the trivial tree, and inverses are given by $[T_-,b,T_+]^{-1}=[T_+,b^{-1},T_-]$.

The elements of $\Vbr$ of the form $[T_-,p,T_+]$ for $p$ a pure braid form a subgroup $\Fbr\le \Vbr$. Since the property of a braid being pure is preserved under bifurcating a strand into a pair of parallel strands, $\Fbr$ is indeed closed under multiplication. Another important subgroup is Thompson's group $F$, which consists of elements of the form $[T_-,1,T_+]$; we will sometimes write $[T_-,T_+]$ for this element.

\subsection{The groups $\Fbr(\ty{w})$}\label{sec:deferred}

In our computation of the $\Sigma^m(\Fbr)$ it will be important to understand a certain family of subgroups of $\Fbr$ denoted $\Fbr(\ty{w})$ for $\ty{w}\in\{\ty{0},\ty{1}\}^*$. To define them, we need some setup. Recall that $\{\ty{0},\ty{1}\}^*$ denotes the set of finite words in the alphabet $\{\ty{0},\ty{1}\}$, so given a (finite rooted binary) tree $T$ we can identify each vertex of $T$ with an element of $\{\ty{0},\ty{1}\}^*$, namely the root is identified with the empty word $\varnothing$ and the left (respectively right) child of the vertex identified with $\ty{w}$ is identified with $\ty{w0}$ (respectively $\ty{w1}$). Call $\ty{w}$ the \emph{address} of the leaf identified with $\ty{w}$. Two addresses are \emph{independent} if neither is a prefix of the other.

\begin{definition}[Vine]
Let $\ty{w}\in\{\ty{0},\ty{1}\}^*$, say with length $n$. The \emph{vine to $\ty{w}$} is the unique tree with $n+1$ leaves that has a leaf with address $\ty{w}$. We denote the vine to $\ty{w}$ by $V(\ty{w})$.
\end{definition}

\begin{definition}[$\ty{w}$-deferred]
Let $\ty{w}\in\{\ty{0},\ty{1}\}^*$. A tree $T$ is called \emph{$\ty{w}$-deferred} if every address of a leaf of $T$ either has $\ty{w}$ as a prefix or else is the address of a leaf of $V(\ty{w})$.
\end{definition}

Every tree is $\varnothing$-deferred. Note that vines are not uniquely determined by addresses, since $V(\ty{w0})=V(\ty{w1})$ for all $\ty{w}$. However, if a tree is both $\ty{w0}$-deferred and $\ty{w1}$-deferred then the only option is that it is the vine $V(\ty{w0})=V(\ty{w1})$.

\begin{definition}[$\Fbr(\ty{w})$]
Let $\ty{w}\in\{\ty{0},\ty{1}\}^*$. The subgroup $\Fbr(\ty{w})\le \Fbr$ is defined to be
\[
\Fbr(\ty{w})\defeq \{[T_-,p,T_+]\mid T_- \text{ and } T_+ \text{ are } \ty{w}\text{-deferred}\}\text{.}
\]
\end{definition}

Note that given elements $[T_-,p,T_+],[U_-,q,U_+]\in \Fbr(\ty{w})$ it is possible to take expansions $[T_-',p',T_+'],[U_-',q',U_+']$ such that $T_+'=U_-'$ and all the trees $T_\pm',U_\pm'$ are still $\ty{w}$-deferred, so $\Fbr(\ty{w})$ really is closed under multiplication. Note however that not every representative $(T_-,p,T_+)$ of an element $[T_-,p,T_+]$ of $\Fbr(\ty{w})$ must feature $\ty{w}$-deferred trees; it is possible that reductions must be performed to get $\ty{w}$-deferred trees. Note that $\Fbr(\varnothing)=\Fbr$.

Our next goal is to see that the $\Fbr(\ty{w})$ admit nice HNN-decompositions in terms of each other. For $\ty{w}\in\{\ty{0},\ty{1}\}^*$ let $x_{\ty{w}}$ denote the element $[V(\ty{w01}),V(\ty{w10})]$ of $F$. In particular for all $n\in\N\cup\{0\}$, the element $x_{\ty{1}^n}$ equals the generator $x_n$ of $F$ from the standard infinite presentation $F=\langle x_0,x_1,\dots\mid x_j x_i = x_i x_{j+1}$ for all $i<j\}$.

\begin{lemma}\label{lem:Fbrn_gens}
Let $\ty{w}\in\{\ty{0},\ty{1}\}^*$. The subgroup $\Fbr(\ty{w})$ of $\Fbr$ is generated by $x_{\ty{w}}$ together with either of $\Fbr(\ty{w0})$ or $\Fbr(\ty{w1})$.
\end{lemma}

\begin{proof}
Let $[T_-,p,T_+]$ be an arbitrary element of $\Fbr(\ty{w})$, with $T_-$ and $T_+$ both $\ty{w}$-deferred. Note that any $\ty{w}$-deferred tree $T$ has a unique leaf with address of the form $\ty{w}\ty{0}^m$ and a unique leaf with address of the form $\ty{w}\ty{1}^n$. Call $m$ the \emph{left depth} of $T$ and $n$ the \emph{right depth}.

Let us now focus on the version of the statement using $\Fbr(\ty{w0})$. If $T_-$ and $T_+$ both have left depth $0$ then they are both the vine to $\ty{w}$. In this case after performing an expansion they both equal the vine to $\ty{w0}$, which is $\ty{w0}$-deferred, so $[T_-,p,T_+]\in \Fbr(\ty{w0})$. Now assume $T_-$ and $T_+$ both have positive left depth (one cannot have left depth $0$ without the other one also having left depth $0$, since they have the same number of leaves). If $T_-$ has left depth $1$ then $T_-$ is already $\ty{w0}$-deferred. Now assume $T_-$ has left depth $m\ge 2$. Then $T_-$ can be obtained from $V(\ty{w01})$ by adding carets, so computing the product $x_{\ty{w}}^{-1}[T_-,p,T_+]$ requires no expansion of $(T_-,p,T_+)$, only of $x_{\ty{w}}^{-1}=[V(\ty{w10}),V(\ty{w01})]$, and we compute that it equals $[T_-',p,T_+]$ for some $T_-'$ that has left depth $m-1$. Continuing in this way, $x_{\ty{w}}^{-(m-1)}[T_-,p,T_+]$ equals $[T_-^{(m-1)},p,T_+]$ for some $T_-^{(m-1)}$ with left depth $1$, so without loss of generality $T_-$ has left depth $1$. By a similar argument if $T_+$ has left depth $m\ge 2$ then after right multiplication by $x_{\ty{w}}^{m-1}$ we can assume $T_+$ has left depth $1$. At this point both $T_-$ and $T_+$ have left depth $1$, so are $\ty{w1}$-deferred, and $[T_-,p,T_+]\in \Fbr(\ty{w0})$.

For the version of the statement using $\Fbr(\ty{w1})$ we do an analogous trick to say that, up to multiplying by powers of $x_{\ty{w}}$ we can assume $T_-$ and $T_+$ have right depth $1$, and then $[T_-,p,T_+]\in \Fbr(\ty{w1})$.
\end{proof}

\begin{lemma}\label{lem:conj}
For any $\ty{w}\in\{\ty{0},\ty{1}\}^*$ we have $\Fbr(\ty{w1})^{x_{\ty{w}}}\le \Fbr(\ty{w11})$ and $\Fbr(\ty{w0})^{x_{\ty{w}}^{-1}}\le \Fbr(\ty{w00})$.
\end{lemma}

\begin{proof}
Let $[T_-,p,T_+]\in \Fbr(\ty{w1})$ with $T_\pm$ two $\ty{w1}$-deferred trees. Let $U_\pm$ be the (possibly trivial) subtree of $T_\pm$ whose root is the vertex of $T_\pm$ with address $\ty{w1}$. Let $V(\ty{w01})\cup U_\pm$ be the tree obtained from $V(\ty{w01})$ by adding $U_\pm$ to the leaf of $V(\ty{w01})$ with address $\ty{w1}$. Let $T_\pm'$ be the tree obtained from $T_\pm$ by adding a caret to the leaf with address $\ty{w0}$. Since $T_\pm$ are both $\ty{w1}$-deferred, these leaves have the same position in the left-to-right numbering of leaves, say they are the $k$th, and let $p'$ be $p$ with its $k$th strand bifurcated, so $[T_-,p,T_+]=[T_-',p',T_+']$. Note that $V(\ty{w01})\cup U_\pm=T_\pm'$. Let $V(\ty{w10})\cup U_\pm$ be the tree obtained from $V(\ty{w10})$ by adding $U_\pm$ to the leaf of $V(\ty{w10})$ with address $\ty{w11}$. Note that $x_{\ty{w}} = [V(\ty{w01}),V(\ty{w10})]$ equals $[V(\ty{w01})\cup U_\pm,V(\ty{w10})\cup U_\pm]$, since if the leaf of $V(\ty{w01})$ with address $\ty{w1}$ is the $\ell$th leaf then the $\ell$th leaf of $V(\ty{w10})$ is the leaf with address $\ty{w11}$. We now compute:
\begin{align*}
x_{\ty{w}}^{-1}[T_-,p,T_+]x_{\ty{w}} &= [V(\ty{w10})\cup U_-,V(\ty{w01})\cup U_-][T_-',p',T_+'][V(\ty{w01})\cup U_+,V(\ty{w10})\cup U_+] \\
&= [V(\ty{w10})\cup U_-,p',V(\ty{w10})\cup U_+]\text{.}
\end{align*}
Since $V(\ty{w10})\cup U_\pm$ are both $\ty{w11}$-deferred, this element lands in $\Fbr(\ty{w11})$. This finishes the proof that $\Fbr(\ty{w1})^{x_{\ty{w}}}\le \Fbr(\ty{w11})$, and the fact that $\Fbr(\ty{w0})^{x_{\ty{w}}^{-1}}\le \Fbr(\ty{w00})$ follows by a parallel argument.
\end{proof}

\begin{corollary}\label{cor:Fbrn_hnn}
For any $\ty{w}\in\{\ty{0},\ty{1}\}^*$, we have that $\Fbr(\ty{w})$ decomposes as a strictly ascending HNN-extension in the following two ways:
\[
\Fbr(\ty{w})=\Fbr(\ty{w0})*_{x_{\ty{w}}^{-1}} \text{ and } \Fbr(\ty{w})=\Fbr(\ty{w1})*_{x_{\ty{w}}}\text{.}
\]
\end{corollary}

\begin{proof}
For the first decomposition, using Citation~\ref{cit:int_hnn} we need to show that $\Fbr(\ty{w0})$ and $x_{\ty{w}}^{-1}$ generate $\Fbr(\ty{w})$, that $\Fbr(\ty{w0})^{x_{\ty{w}}^{-1}}\subseteq \Fbr(\ty{w0})$, and that $\Fbr(\ty{w0})^{x_{\ty{w}}^{-1}}\neq \Fbr(\ty{w0})$. The first claim follows from Lemma~\ref{lem:Fbrn_gens} and the last two claims follow from Lemma~\ref{lem:conj} since $\Fbr(\ty{w00})\subsetneq \Fbr(\ty{w0})$. For the second decomposition, we need to show that $\Fbr(\ty{w1})$ and $x_{\ty{w}}$ generate $\Fbr(\ty{w})$, that $\Fbr(\ty{w1})^{x_{\ty{w}}}\subseteq \Fbr(\ty{w1})$, and that $\Fbr(\ty{w1})^{x_{\ty{w}}}\neq \Fbr(\ty{w1})$, and again these follow from the lemmas.
\end{proof}

As a remark, these HNN-decompositions are reminiscent of those found in \cite{zaremsky16} for the Lodha--Moore groups introduced in \cite{lodha16}, and we suspect there is a more general class of groups for which this behavior occurs, though we will not attempt to develop this here.

It is sometimes notationally cumbersome to prove things about $\Fbr(\ty{w})$ for arbitrary $\ty{w}$, but the following shows that at least up to isomorphism it is enough to understand the $\Fbr(\ty{1}^n)$.

\begin{lemma}\label{lem:iso_deferred}
Let $\ty{w}\in\{\ty{0},\ty{1}\}^*$ have length $n$. Then $\Fbr(\ty{w})\cong \Fbr(\ty{1}^n)$. More precisely $\Fbr(\ty{w})$ and $\Fbr(\ty{1}^n)$ are conjugate in $\Vbr$.
\end{lemma}

\begin{proof}
Say the $k$th leaf of $V(\ty{w})$ is the one with address $\ty{w}$, and note that $V(\ty{w})$ has $n+1$ leaves. Let $b\in B_{n+1}$ be any braid mapping under $\rho$ to the transposition switching $k$ and $n+1$ in $S_{n+1}$. We now claim that
\[
[V(\ty{w}),b,V(\ty{1}^n)]^{-1}\Fbr(\ty{w})[V(\ty{w}),b,V(\ty{1}^n)] = \Fbr(\ty{1}^n) \text{.}
\]
An element of $\Fbr(\ty{w})$ is of the form $[T_-,p,T_+]$ for $T_\pm$ $\ty{w}$-deferred, say with $m$ leaves, and $p\in PB_m$. We can expand $[V(\ty{w}),b,V(\ty{1}^n)]$ in two ways, to get $[T_-,b',U_-]$ and $[T_+,b'',U_+]$ for some $\ty{1}^n$-deferred $U_-$ and $U_+$, and some $b',b''\in PB_m$. Note that $b'$ equals $b$ with its last strand replaced by $m-n$ parallel strands, and $b''$ does as well, so actually $b''=b'$. Now we have
\begin{align*}
[V(\ty{w}),b,V(\ty{1}^n)]^{-1}[T_-,p,T_+][V(\ty{w}),b,V(\ty{1}^n)] &= [T_-,b',U_-]^{-1}[T_-,p,T_+][T_+,b',U_+] \\
&= [U_-,(b')^{-1},T_-][T_-,p,T_+][T_+,b',U_+] \\
&= [U_-,(b')^{-1}pb',U_+]\text{.}
\end{align*}
Since $U_\pm$ are $\ty{1}^n$-deferred and the braid $(b')^{-1}pb'$ is pure, this indeed lies in $\Fbr(\ty{1}^n)$. This shows one inclusion of our claimed equality, and a parallel argument gives the other inclusion.
\end{proof}

We close this section with an observation that will not have any direct consequences later but we find to be intriguing in its own right. Since all the $F(\ty{w}) \defeq F\cap \Fbr(\ty{w})$ are known to be isomorphic to $F$, one might guess that all the $\Fbr(\ty{w})$ are isomorphic to $\Fbr$. However, this is not true:

\begin{lemma}\label{lem:quotient_PB}
For any $\ty{w}\in\{\ty{0},\ty{1}\}^*$ with length $n$, the group $\Fbr(\ty{w})$ surjects onto $PB_n$.
\end{lemma}

\begin{proof}
By Lemma~\ref{lem:iso_deferred} we can assume $\ty{w}=\ty{1}^n$. For $m\ge n$ let $\phi_n\colon PB_m\to PB_n$ be the projection given by deleting all but the first $n$ strands of a braid (we use the same notation $\phi_n$ for any $m$). Given an element $[T_-,p,T_+]$ of $\Fbr(\ty{1}^n)$ call $(T_-,p,T_+)$ a \emph{$\ty{1}^n$-deferred representative} of its class $[T_-,p,T_+]$ if $T_\pm$ are $\ty{1}^n$-deferred (so being in $\Fbr(\ty{1}^n)$ is equivalent to admitting a $\ty{1}^n$-deferred representative). Now define $\psi \colon \Fbr(\ty{1}^n) \to PB_n$ by sending $[T_-,p,T_+]$ to $\phi_n(p)$, where $(T_-,p,T_+)$ is a $\ty{1}^n$-deferred representative of its class. We need to check that this is a well defined homomorphism, and then we will be done since it is obviously surjective. Suppose $(T_-,p,T_+)$ and $(U_-,q,U_+)$ are $\ty{1}^n$-deferred representatives of the same class. Then they have a common expansion $(T_-',p',T_+')=(U_-',q',U_+')$. Since all of $T_\pm$ and $U_\pm$ all $\ty{1}^n$-deferred, without loss of generality all of $T_\pm'$ and $U_\pm'$ are $\ty{1}^n$-deferred. In particular none of the first $n$ strands of $p$ were bifurcated in moving from $p$ to $p'$, and similarly none of the first $n$ strands of $q$ were bifurcated to get $q'$. The fact that $p'=q'$ thus implies that $\phi_n(p)=\phi_n(p')=\phi_n(q')=\phi_n(q)$. This shows that $\psi$ is well defined. The fact that it is a homomorphism now follows simply because $\phi_n$ is a homomorphism.
\end{proof}

\begin{corollary}
For any $\ty{w}\in\{\ty{0},\ty{1}\}^*$ with length at least $3$, the group $\Fbr(\ty{w})$ is not isomorphic to $\Fbr$, nor to any quotient of $\Fbr$.
\end{corollary}

\begin{proof}
By \cite[Theorem~2.1]{zaremsky18} every non-abelian quotient of $\Fbr$ contains a copy of Thompson's group $F$. By Lemma~\ref{lem:quotient_PB} $\Fbr(\ty{w})$ surjects onto $PB_n$, which is non-abelian since $n\ge 3$, and does not contain $F$ (for many obvious reasons, such as being finite dimensional, linear, residually finite, and so forth).
\end{proof}

It is unclear whether $\Fbr(\ty{w})$ for $\ty{w}$ of length $1$ or $2$ is isomorphic to $\Fbr$, but we expect not. In fact we believe the length of $\ty{w}$ precisely determines the isomorphism type of $\Fbr(\ty{w})$, by virtue of the rank of the abelianization of $\Fbr(\ty{w})$ being different for different lengths of $\ty{w}$, but we will leave the verification of this for the future.

\section{The cube complexes}\label{sec:cubes}

In \cite{bux16} cube complexes were introduced on which $\Vbr$ and $\Fbr$ act nicely. There the term ``Stein space'' was used, and more recently the term ``Stein--Farley complex'' has become common, in reference to seminal work of Stein \cite{stein92} and Farley \cite{farley03} on the analogous complexes for Thompson's groups $F$, $T$ and $V$. These cube complexes have all been used to prove their respective groups are of type $\F_\infty$.

Here we will primarily be interested in the Stein--Farley complex for $\Fbr$, and so we will denote this one by $X$. The Stein--Farley complex for $\Vbr$ will be denoted $X(\Vbr)$, and will only be discussed briefly in Remark~\ref{rmk:Vbr_CAT0}. We caution that in \cite{bux16}, $X$ denoted the complex for $\Vbr$ and $X(\Fbr)$ denoted the complex for $\Fbr$. We will use (what we are calling) $X$ for two main purposes. The first is establishing a family of arbitrarily highly connected subcomplexes on which $\Fbr$ acts cocompactly (in Lemma~\ref{lem:stein_truncate}), which a glance at Definition~\ref{def:bnsr} shows is crucial when computing BNSR-invariants. The second is to establish subcomplexes $X(n)$ on which the $\Fbr(\ty{1}^n)$ act nicely, and which will be used in the next section to prove all the $\Fbr(\ty{w})$ are of type $\F_\infty$.

\subsection{The construction}\label{sec:construct_stein}

Before we recall the construction of the Stein--Farley complex for $\Fbr$ we need some setup. By a \emph{forest} we will always mean a finite sequence of (finite rooted binary) trees, whose roots and leaves comprise the \emph{roots} and \emph{leaves} of the forest. The notions of expansion and reduction for triples $(T_-,p,T_+)$ make sense also for triples $(F_-,p,F_+)$, where $F_\pm$ are forests. We write $[F_-,p,F_+]$ for the equivalence class of $(F_-,p,F_+)$ under expansions and reductions. Since two such triples $(F_-,p,F_+)$ and $(E_-,q,E_+)$ admit expansions $(F_-',p',F_+')$ and $(E_-',q',E_+')$ with $F_+'=E_-'$ if and only if $F_+$ and $E_-$ have the same number of roots, in this and only this case it makes sense to take the product
\[
[F_-,p,F_+][E_-,q,E_+] = [F_-',p'q',E_+']\text{.}
\]
In this way the set of equivalence classes $[F_-,p,F_+]$ has a groupoid structure.

For each $n$, there is a copy of $PB_n$ inside this groupoid, namely $\{[1_n,p,1_n]\mid p\in PB_n\}$, where $1_n$ is the trivial forest with $n$ roots. This acts from the right on the set $\{[F_-,p,F_+]\mid F_+$ has $n$ roots$\}$, and it makes sense to consider the orbit space, the elements of which are groupoid cosets. Write $[F_-,p,F_+]_{PB}$ for the coset $[F_-,p,F_+]\{[1_n,p,1_n]\mid p\in PB_n\}$, where $n$ is the number of roots of $F_+$. When we do not need or want to specify a number of roots for a trivial forest we will sometimes write $1$ instead of $1_n$.

\begin{definition}[Elementary]
Call a forest \emph{elementary} if each of its trees is either trivial or a single caret. For $E$ elementary call $[E,1,1]$ an \emph{elementary split} and $[1,1,E]$ an \emph{elementary merge}.
\end{definition}

\begin{definition}[Stein--Farley complex]
The \emph{Stein--Farley complex} for $\Fbr$, which we denote by $X$, is the cube complex constructed as follows. The $0$-cubes are the equivalence classes $[T,p,F]_{PB}$ as above, with $T$ a tree. The $1$-cubes are given by the rule that $[T,p,F]_{PB}$ and $[T,p,F][E,1,1]_{PB}$ span a $1$-cube whenever $E$ is an elementary forest with exactly one non-trivial tree (and the same number of roots as $F$). Now let $[T,p,F]_{PB}$ be a $0$-cube and $E$ an elementary forest with the same number of roots as $F$. Say $E$ has $k$ non-trivial trees. The set of $0$-cubes of the form $[T,p,F][E',1,1]_{PB}$ for $E'$ a subforest of $E$ span the $1$-skeleton of a $k$-cube, and we declare that they span a $k$-cube in $X$.
\end{definition}

\begin{cit}\cite[Section~2]{bux16}
$X$ is contractible.
\end{cit}

In Theorem~\ref{thrm:X_CAT0} we will see that it is even $\CAT(0)$.

Note that $\Fbr$ acts cellularly on $X$ from the left. There is a natural $\Fbr$-invariant function from the set of $0$-cubes of $X$ to $\N$ given by sending $[T,p,F]_{PB}$ to the number of roots of $F$. This function can be extended to an affine function on each cube of $X$ that is non-constant on positive-dimensional cubes, which induces a Morse function
\[
h\colon X \to \R \text{.}
\]
Note that the action of $\Fbr$ on $X$ preserves $h$-values. By the $\Fbr$ version of \cite[Lemma~2.5]{bux16}, the action of $\Fbr$ on $X^{n\le h\le m}$ is cocompact for all $n\le m<\infty$.

Each cube in $X$ has a unique $0$-face at which $h$ is maximized, which we will call the \emph{top} of the cube, and a unique $0$-face at which $h$ is minimized, called the \emph{bottom}. The ascending link $\alk x$ of a $0$-cube $x$ consists of all the directions into cubes having $x$ as their bottom, and the descending link $\dlk x$ consists of all directions into cubes with $x$ as their top.

\begin{lemma}\label{lem:alk_cible}
For any $0$-cube $x$ in $X$, say with $h(x)=k$, the ascending link $\alk x$ is a $(k-1)$-simplex, hence is contractible.
\end{lemma}

\begin{proof}
There is a unique elementary forest with $k$ roots such that every elementary forest with $k$ roots is a subforest of that one; it has $k$ carets. This implies that there is a unique cube with $x$ as its bottom $0$-face such that every cube with $x$ as its bottom $0$-face is a face of that one; it is $k$-dimensional. We conclude that $\alk x$ is a $(k-1)$-simplex, hence contractible.
\end{proof}

Intuitively, $\alk x$ is all the ways of adding splits to the ``feet'' of $x$, so the rule ``split every foot'' gives this unique maximal simplex. As a remark, this does not work for $\dlk x$ because of the modding out of the actions of the $PB_n$; ``merge every foot'' is not well defined since an arbitrary amount of braiding is allowed first. Indeed, $\dlk x$ is not contractible, but \cite[Corollary~4.4,Corollary~4.5]{bux16} says that it does get arbitrarily highly connected as $h(x)$ tends to infinity. We use this fact to prove the following:

\begin{lemma}\label{lem:stein_truncate}
For any $m$, the complex $X^{2m\le h\le 4m}$ is $(m-1)$-connected.
\end{lemma}

\begin{proof}
The ascending link of any $0$-cube is contractible (Lemma~\ref{lem:alk_cible}). Since $X$ is contractible, Lemma~\ref{lem:morse} implies $X^{p\le h}$ is contractible for any $p$; in particular $X^{2m\le h}$ is contractible. Now let $x$ be a $0$-cube in $X$ with $h(x)>4m$. The descending link of $x$ in $X$ lies entirely in $X^{2m\le h}$, since an elementary forest describing a merge of $x$ can feature at most $h(x)/2$ carets. Hence the descending link $\dlk_{X^{2m\le h}}x$ equals $\dlk_X x$, and since $h(x)>4m$ \cite[Corollary~4.5]{bux16} says this is $(m-1)$-connected. Now Lemma~\ref{lem:morse} says that $X^{2m\le h\le 4m}$ is $(m-1)$-connected.
\end{proof}

\subsection{The complexes $X(n)$}\label{sec:Xn}

A crucial fact we will need to compute the $\Sigma^m(\Fbr)$ is that the $\Fbr(\ty{1}^n)$, and hence all the $\Fbr(\ty{w})$, are of type $\F_\infty$. To do this we first establish an $\Fbr(\ty{1}^n)$-stable subcomplex $X(n)$ of $X$, and then use it in the next section to prove that $\Fbr(\ty{1}^n)$ is of type $\F_\infty$. For brevity we will write $\Fbr(n)$ for $\Fbr(\ty{1}^n)$ and use the term \emph{$n$-deferred} to mean $\ty{1}^n$-deferred. Also, call a forest \emph{$n$-bare} if its first $n$ trees are trivial.

\begin{definition}
Let $X(n)$ be the full subcomplex of $X$ supported on those $0$-cubes $[T,p,F]_{PB}$ such that $T$ is $n$-deferred and $F$ is $n$-bare.
\end{definition}

As a remark, every $0$-cube $x$ in $X(n)$ satisfies $h(x)\ge n+1$. This is because an $n$-bare forest has at least $n$ roots and has exactly $n$ roots only if it is trivial, but $n$-deferred trees have at least $n+1$ leaves.

\begin{lemma}\label{lem:Xn_stable}
The action of $\Fbr(n)$ on $X$ stabilizes $X(n)$.
\end{lemma}

\begin{proof}
Let $[T_-,p,T_+]\in \Fbr(n)$ and $[T,q,F]_{PB}\in X(n)^{(0)}$, say $T_-,T_+,T$ are $n$-deferred trees and $F$ is $n$-bare. We need to show that $[T_-,p,T_+][T,q,F]_{PB}\in X(n)^{(0)}$. Since $T_+$ and $T$ are $n$-deferred, we can perform expansions $[T_-,p,T_+]=[T_-',p',T_+']$ and $[T,q,F]=[T',q',F']$ such that $T_+'=T'$ and such that in the course of doing the expansions we never add carets to any of the first $n$ leaves of the relevant trees or forests. In particular $T_-'$ is $n$-deferred and $F'$ is $n$-bare, so $[T_-,p,T_+][T,q,F]_{PB} = [T_-',p'q',F']_{PB}$ does indeed lie in $X(n)^{(0)}$.
\end{proof}

\begin{lemma}\label{lem:Fbrn_cocpt}
The action of $\Fbr(n)$ is cocompact on each $X(n)^{h\le q}$ for $q<\infty$.
\end{lemma}

\begin{proof}
First note that if $[T,p,F]_{PB}$ and $[U,q,E]_{PB}$ are any two $0$-cubes of $X(n)$ such that $F$ and $E$ have the same number of roots, then the element $[T,p,F][E,q,U]$ of $\Fbr$ takes $[U,q,E]_{PB}$ to $[T,p,F]_{PB}$. It is straightforward to see that this element lies in $\Fbr(n)$. We conclude that the action of $\Fbr(n)$ on $X(n)^{h\le q}$ has finitely many orbits of $0$-cubes. Finally, observe that for a $0$-cube $x$ of $X$ there are only finitely many cubes in $X$ with $x$ as their bottom $0$-face, so indeed the action of $\Fbr(n)$ on $X(n)^{h\le q}$ is cocompact.
\end{proof}

\begin{lemma}\label{lem:Fbrn_stabs}
Every stabilizer in $\Fbr(n)$ of a cube in $X(n)$ is isomorphic to a pure braid group, and hence is of type $\F_\infty$.
\end{lemma}

\begin{proof}
For $x=[T,p,F]_{PB}$ a $0$-cube in $X(n)$, say with $h(x)=m$, the (pure analog of the) proof of \cite[Lemma~2.6]{bux16} says that
\[
\Stab_{\Fbr}(x)=\{[T,p,F][1,q,1][F,p^{-1},T]\mid q\in PB_m\} \cong PB_m\text{.}
\]
Note that $[T,p,F][1,q,1][F,p^{-1},T] = [T,pq'p^{-1},T]$ for some $q'$, so if we assume $T$ is $n$-deferred (which is possible since $[T,p,F]_{PB}\in X(n)$) then we conclude that $\Stab_{\Fbr}(x) \le \Fbr(n)$. In particular the stabilizer of $x$ in $\Fbr(n)$ equals the stabilizer of $x$ in $\Fbr$, so is isomorphic to $PB_m$. For positive dimensional cubes, we note that a consequence of the pure version of \cite[Corollary~2.8]{bux16} is that the stabilizer in $\Fbr$ of a cube equals the stabilizer in $\Fbr$ of its bottom $0$-face.
\end{proof}

Looking at the action of $\Fbr(n)$ on $X(n)$, the last things to prove to get $\Fbr(n)$ to be of type $\F_\infty$ are that descending links in $X(n)$ are increasingly highly connected, and $X(n)$ is contractible. These are the goals of the next section.

\section{Arc complexes and descending links in $X$}\label{sec:arc}

Let $S$ be a compact connected oriented surface with (possibly empty) boundary $\partial S$, and $P$ a finite set of marked points in $S\setminus \partial S$. For our purposes an \emph{arc} is a simple path in $S\setminus \partial S$ that meets $P$ precisely at its endpoints. If the endpoints of an arc coincide then we additionally require it to not bound a disk in $S\setminus P$.

Often in the literature, instead of specifying a set of marked points as the candidates for the endpoints of arcs, the boundary $\partial S$ serves as the set of such candidates. One can pass between these viewpoints by replacing marked points with boundary components, and as explained in \cite[Section~5.3.1]{farb12} from the point of view of the arc complex, defined below, there is no difference between these viewpoints (more precisely they yield isomorphic notions of arc complex). Here we will use marked points, as was done in \cite{bux16}, since later we will be identifying them with vertices of graphs. Note however that our surface $S$ does still have boundary $\partial S$, which arcs are not permitted to touch.

\subsection{The arc complex and variations}\label{sec:arc_cpx}

\begin{definition}[Arc system/complex]
An \emph{arc system} is a set of arcs $\{\alpha_0,\dots,\alpha_k\}$ such that for each $i\ne j$, the arcs $\alpha_i$ and $\alpha_j$ are disjoint except possibly at their endpoints and represent distinct homotopy classes relative $P$. The homotopy classes relative $P$ of arc systems form the simplices of a simplicial complex called the \emph{arc complex (on $(S,P)$)}, denoted $\arcs$. The face relation for simplices is given by passing to subsets of representative arc systems.
\end{definition}

If $\{\alpha,\beta\}$ represents a $1$-simplex in $\arcs$ call the arcs $\alpha$ and $\beta$ \emph{compatible}.

\begin{definition}[Modeled on $\Gamma$]\label{def:modeled}
Let $\Gamma$ be a graph with $|P|$ vertices and fix an identification $V(\Gamma)\leftrightarrow P$. We say an arc $\alpha$ is \emph{modeled on $\Gamma$} if the vertices of $\Gamma$ corresponding to the endpoints of $\alpha$ are the endpoints of an edge in $\Gamma$. Write $\arcs(\Gamma)$ for the subcomplex of $\arcs$ consisting of arc systems whose arcs are all modeled on $\Gamma$.
\end{definition}

Note that if $\alpha$ is modeled on $\Gamma$ and the endpoints of $\alpha$ coincide, say at $p\in P$, then the vertex of $\Gamma$ corresponding to $p$ must have a loop at it.

\begin{definition}[Matching arc complex]
An arc system $\{\alpha_0,\dots,\alpha_k\}$ is called a \emph{matching arc system} if for each $i$ the arc $\alpha_i$ has distinct endpoints and for each $i\ne j$ the arcs $\alpha_i$ and $\alpha_j$ are disjoint including at their endpoints. The \emph{complete matching arc complex (on $(S,P)$)} $\matcharc$ is the subcomplex of $\arcs$ consisting of homotopy classes of matching arc systems. The \emph{matching arc complex (on $(S,P)$) modeled on $\Gamma$} is $\matcharc(\Gamma)\defeq \matcharc \cap \arcs(\Gamma)$.
\end{definition}

If $\{\alpha,\beta\}$ represents a $1$-simplex in $\matcharc$ call the arcs $\alpha$ and $\beta$ \emph{matching-compatible}; see Figure~\ref{fig:arcs}.

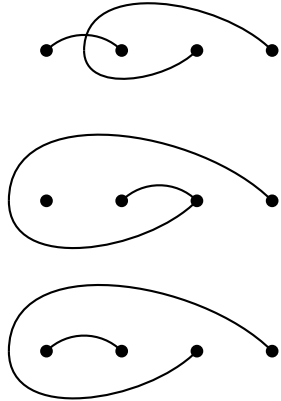
\begin{figure}[hbt]
 \begin{tikzpicture}[line width=0.8pt]
  \draw (1,0) to [out=45, in=135] (2,0);
	\draw (3,0) to [out=225, in=-90] (1.5,0); \draw (1.5,0) to [out=90, in=135] (4,0);
	\filldraw (1,0) circle (2pt)   (2,0) circle (2pt)   (3,0) circle (2pt)   (4,0) circle (2pt);
	
	\begin{scope}[yshift=-2cm]
	\draw (2,0) to [out=45, in=135] (3,0);
	\draw (3,0) to [out=225, in=-90] (0.5,0); \draw (0.5,0) to [out=90, in=135] (4,0);
	\filldraw (1,0) circle (2pt)   (2,0) circle (2pt)   (3,0) circle (2pt)   (4,0) circle (2pt);
	\end{scope}
	
	\begin{scope}[yshift=-4cm]
	\draw (1,0) to [out=45, in=135] (2,0);
	\draw (3,0) to [out=225, in=-90] (0.5,0); \draw (0.5,0) to [out=90, in=135] (4,0);
	\filldraw (1,0) circle (2pt)   (2,0) circle (2pt)   (3,0) circle (2pt)   (4,0) circle (2pt);
	\end{scope}
	
 \end{tikzpicture}
 \caption{From top to bottom, a pair of non-compatible arcs, a pair of compatible but not matching-compatible arcs, and a pair of matching-compatible arcs.}\label{fig:arcs}
\end{figure}

The terminology of matchings was inspired by the analogous notion for collections of edges in a graph, which form a \emph{matching} when they are pairwise disjoint. Note that $\matcharc=\matcharc(K_{|P|})$ for $K_{|P|}$ the complete graph on $|P|$ vertices, hence the name. In general $\matcharc(\Gamma)=\matcharc(\Gamma')$ where $\Gamma'$ is $\Gamma$ with any loops removed.

Call a graph \emph{linear} if its vertices can be identified with $\{1,\dots,n\}$ and its edges with $\{\{i,i+1\}\mid 1\le i\le n-1\}$ for some $n$.

\begin{cit}\cite[Theorem~3.10]{bux16}\label{cit:matcharc_conn}
If $\Gamma$ is a subgraph of a linear graph, then $\matcharc(\Gamma)$ is $\left(\left\lfloor\frac{|E(\Gamma)|-1}{4}\right\rfloor-1\right)$-connected.
\end{cit}

\subsection{Descending links}\label{sec:dlks}

The connection between matching arc complexes and descending links of $0$-cubes in $X$ is as follows. Let $x=[T,p,F]_{PB}$ be a $0$-cube in $X$ with $h(x)=m$. The simplices in $\dlk x$ are in one to one correspondence with the positive dimensional cubes in $X$ having $x$ as their top. Each such cube is uniquely determined by its bottom $0$-cube, which is necessarily of the form $[T,p,F][1,q,E]_{PB}$ for some $q\in PB_m$ and some elementary forest $E$ with $m$ leaves. We will notationally represent this simplex of $\dlk x$ by $[1,q,E]_{PB}$. If the simplex is $k$-dimensional then $E$ has $k+1$ non-trivial trees. Now let $\Gamma_m$ be the linear graph with $m$ vertices. Let $S$ be a disk and $P$ a set of $m$ points in $S$. Fix an embedding of $\Gamma_m$ into $S$, which in particular gives us an identification of $V(\Gamma_m)$ with $P$ as in Definition~\ref{def:modeled}. Now matchings of $\Gamma_m$ (meaning collections of disjoint edges) can be viewed as simplices in $\matcharc(\Gamma_m)$. Note that $PB_m$ acts on $\matcharc(\Gamma_m)$ since it is the pure mapping class group of the disk with $m$ marked points. As discussed in \cite[Section~4]{bux16}, there is a surjective simplicial map $\pi\colon \dlk x \to \matcharc(\Gamma_m)$ given by
\[
\pi \colon [1,p,E]_{PB} \mapsto p(\Gamma_E)\text{,}
\]
where $\Gamma_E$ is the matching in $\Gamma_m$ whose edges are those pairs $\{i,i+1\}$ such that the $i$th and $(i+1)$st leaves of $E$ share a caret. (Again, we view $\Gamma_E$ as both a matching in $\Gamma_m$ and a matching arc system.) This map is not injective, but the fiber of any $k$-simplex is $(k-1)$-connected. Links in $\matcharc(\Gamma_m)$ are similarly highly connected, so \cite[Theorem~9.1]{quillen78} says that $\dlk x$ is at least as highly connected as $\matcharc(\Gamma_m)$, and hence Citation~\ref{cit:matcharc_conn} tells us that as $h(x)$ goes to infinity, $\dlk x$ becomes arbitrarily highly connected.

Working through an analogous argument for $X(n)$, we get the following:

\begin{proposition}\label{prop:Xn_dlk}
Let $x$ be a $0$-cube in $X(n)$, say with $h(x)=m$ (so $n+1\le m$). The descending link of $x$ in $X(n)$ is $\left(\left\lfloor\frac{m-n-2}{4}\right\rfloor-1\right)$-connected.
\end{proposition}

\begin{proof}
Let $\Gamma_m(n)$ be the subgraph of $\Gamma_m$ with the same vertex set as $\Gamma_m$ and whose edges are given by the pairs $\{i,i+1\}$ with $i\ge n+1$. Since $\Gamma_m(n)$ has $m-n-1$ edges, Citation~\ref{cit:matcharc_conn} says $\matcharc(\Gamma_m(n))$ is $\left(\left\lfloor\frac{m-n-2}{4}\right\rfloor-1\right)$-connected. The map $\pi \colon \dlk_X x \to \matcharc(\Gamma_m)$ above restricts to a surjective simplicial map $\pi' \colon \dlk_{X(n)}x \to \matcharc(\Gamma_m(n))$. We want to apply \cite[Theorem~9.1]{quillen78}, so we need to check high connectivity of fibers and links. For any $k$-simplex $\sigma$ in $\matcharc(\Gamma_m(n))$, the fiber $(\pi')^{-1}(\sigma)$ equals the full fiber $\pi^{-1}(\sigma)$, hence is $(k-1)$-connected by \cite[Proposition~4.3]{bux16}. The link of $\sigma$ in $\matcharc(\Gamma_m(n))$ is isomorphic to $\matcharc(\Delta)$ for some subgraph $\Delta$ of $\Gamma_m(n)$ with at least $m-n-1-3(k+1)$ edges, so is $\left(\left\lfloor\frac{m-n-2-3(k+1)}{4}\right\rfloor-1\right)$-connected by Citation~\ref{cit:matcharc_conn}, and hence $\left(\left\lfloor\frac{m-n-2}{4}\right\rfloor-k-2\right)$-connected. Now \cite[Theorem~9.1]{quillen78} gives us our conclusion.
\end{proof}

To prove that $X(n)$ is contractible we will prove that it is $\CAT(0)$. First we will prove that $X$ is $\CAT(0)$, which should not be surprising to the experts, but has not been stated or proved in the literature before. For the classical Thompson groups, the $\CAT(0)$ property for the Stein--Farley complexes was one of the main results of \cite{farley03}.

\begin{theorem}\label{thrm:X_CAT0}
The Stein--Farley complex $X$ is $\CAT(0)$.
\end{theorem}

\begin{proof}
Since $X$ is simply connected (in fact contractible), we just need to check that links of $0$-cubes are flag complexes. Let $x$ be a $0$-cube, say with $h(x)=m$. First we claim that the descending link $\dlk x$ is flag. Note that $\matcharc(\Gamma_m)$ is flag by the proof of \cite[Lemma~3.2]{bux16}. Now take a collection of $0$-simplices $y_0,\dots,y_k$ in $\dlk x$ that pairwise span $1$-simplices, so their images under $\pi$ in $\matcharc(\Gamma_m)$ pairwise span $1$-simplices and hence span a simplex. We can inductively assume that $y_1,\dots,y_k$ span a simplex, and so the pure version of \cite[Lemma~4.2]{bux16} implies that $y_0,\dots,y_k$ span a simplex, as desired. This shows that $\dlk x$ is flag. Note also that, by Lemma~\ref{lem:alk_cible}, the ascending link $\alk x$ is a simplex, hence is flag.

Now consider the link $\lk x$, and take a collection of $0$-simplices in $\lk x$ that pairwise span $1$-simplices. Let the collection be $y_0,\dots,y_k,z_{k+1},\dots,z_\ell$, with all $y_i\in \dlk x$ and all $z_i\in \alk x$ (these are the only options for $0$-simplices of $\lk x$). By abuse of notation also write $y_i$ for the $0$-cube in $X$ that shares a $1$-cube with $x$ and corresponds to the direction $y_i$ in $\lk x$ (and similarly for $z_i$). Since $\dlk x$ and $\alk x$ are flag, we can choose $0$-cubes $y$ and $z$ such that there is a cube with top $z$ and bottom $x$ containing each $z_i$, and also a cube with top $x$ and bottom $y$ containing each $y_i$. It suffices to show that there is a cube with top $z$ and bottom $y$ containing $x$, since then by the construction of $X$ it will have these as faces.

Write $x=[T,p,F]_{PB}$, and choose elementary forests $E_0,\dots,E_k,D_{k+1},\dots,D_\ell$ and pure braids $q_0,\dots,q_k$ such that $y_i=[T,p,F][1,q_i,E_i]_{PB}$ and $z_i=[T,p,F][D_i,1,1]_{PB}$ for all $i$. For an elementary forest $E$ number the roots and the leaves from left to right, define the \emph{root support} of $E$ to be $\{i\mid$ the $i$th root of $E$ belongs to a non-trivial tree$\}$ and the \emph{leaf support} of $E$ to be $\{i\mid$ the $i$th leaf of $E$ belongs to a non-trivial tree$\}$. The fact that $y_0,\dots,y_k,z_{k+1},\dots,z_\ell$ pairwise span $1$-simplices in $\lk x$ implies that the leaf supports of the $E_i$ are pairwise disjoint, as are the root supports of the $D_i$, and each leaf support of an $E_i$ is disjoint from every root support of a $D_j$. The $0$-cube $y$ is of the form $[T,p,F][1,q,E]_{PB}$ for some $q$ and some $E$ whose leaf support equals the union of the leaf supports of the $E_i$. The $0$-cube $z$ is of the form $[T,p,F][D,1,1]{PB}$ for some $D$ whose root support equals the union of the root supports of the $D_i$. In particular the leaf support of $E$ is disjoint from the root support of $D$. This tells us that the product $[E,1,1][D,1,1]$ equals $[C,1,1]$ for $C$ an elementary forest. The equation
\begin{align*}
y[C,1,1]_{PB} &= [T,p,F][1,q,E][C,1,1]_{PB} \\
&= [T,p,F][1,q,E][E,1,1][D,1,1]_{PB} \\
&= [T,p,F][D,1,1]_{PB} \\
&= z
\end{align*}
thus finishes the proof.
\end{proof}

\begin{corollary}\label{cor:Xn_cible}
$X(n)$ is $\CAT(0)$ and hence contractible.
\end{corollary}

\begin{proof}
First note that $X(n)$ is connected, since the $0$-cube $[T,p,F]_{PB}$ in $X(n)$ is connected by a path in $X(n)$ to $[T,1,1]_{PB}$, and any $0$-cubes in $X(n)$ of the form $[T_1,1,1]_{PB},[T_2,1,1]_{PB}$ are connected by a path in $X(n)$. Now it suffices by Citation~\ref{cit:loc_cvx} to prove that for any $0$-cube $x$ in $X(n)$, the link $\lk_{X(n)}x$ of $x$ in $X(n)$ is a full subcomplex of the link $\lk_X x$ of $x$ in $X$. Let $y_0,\dots,y_k,z_{k+1},\dots,z_\ell$ be a collection of $0$-simplices in $\lk_X x$ that span a simplex in $\lk_X x$, as in the proof of Theorem~\ref{thrm:X_CAT0}, and assume they even lie in $\lk_{X(n)} x$. Then the leaf supports of the $E_i$ and the root supports of the $D_i$ (defined in the proof of Theorem~\ref{thrm:X_CAT0}) must all be disjoint from $\{1,\dots,n\}$. This implies that the leaf support of the forest $E$ and the root supports of the forests $D$ and $C$ (defined in the proof of Theorem~\ref{thrm:X_CAT0}) are all disjoint from $\{1,\dots,n\}$, and hence all the cubes built in the proof of Theorem~\ref{thrm:X_CAT0} actually lie in $X(n)$. We conclude that the simplex spanned by $y_0,\dots,y_k,z_{k+1},\dots,z_\ell$ lies in $\lk_{X(n)} x$, as desired.
\end{proof}

\begin{corollary}\label{cor:Fbrn_Finfty}
For any $n$, $\Fbr(n)$ is of type $\F_\infty$. Hence for any $\ty{w}\in\{\ty{0},\ty{1}\}^*$, $\Fbr(\ty{w})$ is of type $\F_\infty$.
\end{corollary}

\begin{proof}
Consider the cellular action of $\Fbr(n)$ on $X(n)$. We claim that all the conditions of Lemma~\ref{lem:brown} are met. First, $X(n)$ is contractible by Corollary~\ref{cor:Xn_cible}. Moreover, the sublevel sets $X(n)^{h\le q}$ are $\Fbr(n)$-cocompact by Lemma~\ref{lem:Fbrn_cocpt} and the stabilizers in $\Fbr(n)$ of cubes in $X(n)$ are of type $\F_\infty$ by Lemma~\ref{lem:Fbrn_stabs}. Finally, the condition on descending links in Lemma~\ref{lem:brown} follows from Proposition~\ref{prop:Xn_dlk}. The result about $\Fbr(\ty{w})$ is immediate from Lemma~\ref{lem:iso_deferred}.
\end{proof}

\begin{remark}\label{rmk:Vbr_CAT0}
Since $\Vbr$ is perfect \cite{zaremsky18} its BNSR-invariants are all empty, so we are not especially interested in $\Vbr$ at the moment. It is worth mentioning here that the Stein--Farley complex $X(\Vbr)$ for $\Vbr$ is also $\CAT(0)$. This was known to the authors of \cite{bux16}, but was not recorded there because it was not strictly relevant. The proof proceeds analogously to the above proof that $X$ is $\CAT(0)$. See \cite{bux16} for the construction of $X(\Vbr)$ (there denoted $X$), and we leave it to the reader to check that its links are flag, by working through all the results analogous to the above ones for $X$.
\end{remark}

\section{The BNSR-invariants of $\Fbr$}\label{sec:bnsr_BF}

With the HNN-decompositions from Corollary~\ref{cor:Fbrn_hnn} in hand, and having proved in Corollary~\ref{cor:Fbrn_Finfty} that all the $\Fbr(\ty{w})$ are of type $\F_\infty$, we are now ready to compute all the $\Sigma^m(\Fbr)$. The first step is of course to figure out the character sphere $\Sigma(\Fbr)$. It is an easy exercise to abelianize the presentation for $\Fbr$ given in \cite[Theorem~5.2]{brady08} and see that we get $\Z^4$. Hence $\Hom(\Fbr,\R)\cong \R^4$ and $\Sigma(\Fbr)\cong S^3$. More precisely, a basis for $\Hom(\Fbr,\R)$ was given in \cite{zaremsky18}, described as follows. For a tree $T$ let $L(T)$ be the distance from the first leaf to the root, and let $R(T)$ be the distance from the last leaf to the root. For a pure braid $p\in PB_n$ recall that $\omega_{i,j}(p)$ is the winding number of the $i$th and $j$th strands of $p$. A basis for $\Hom(\Fbr,\R)$ is given by the four characters
\begin{align*}
\phi_{\ty{0}} \colon [T_-,p,T_+] &\mapsto L(T_+)-L(T_-)\\
\phi_{\ty{1}} \colon [T_-,p,T_+] &\mapsto R(T_+)-R(T_-)\\
\omega_0 \colon [T_-,p,T_+] &\mapsto \omega_{1,n}(p) \text{ if } p\in PB_n\\
\omega_1 \colon [T_-,p,T_+] &\mapsto \sum_{i=1}^{n-1} \omega_{i,i+1}(p) \text{ if } p\in PB_n \text{.}
\end{align*}

In \cite{zaremsky18} it was proved that
\[
\Sigma^1(\Fbr)=\Sigma(\Fbr)\setminus \{[\phi_{\ty{0}}],[\phi_{\ty{1}}]\}\text{.}
\]
Our main result is:

\begin{theorem}\label{thrm:BNSR_BF}
For $m\ge 2$, $\Sigma^m(\Fbr)$ consists of all points on the sphere $\Sigma(\Fbr)=S^3$ except for those in the convex hull of $[\phi_{\ty{0}}]$ and $[\phi_{\ty{1}}]$.
\end{theorem}

\begin{proof}
Fix $m\ge 2$. Let $[\chi]\in \Sigma(\Fbr)$, say
\[
\chi=a\phi_{\ty{0}} + b\phi_{\ty{1}} + c\omega_0 + d\omega_1 \text{.}
\]
We will consider a variety of cases, for $[\chi]$ lying in various regions of $\Sigma(\Fbr)$. Different cases will require quite different tools.

\textbf{Case 1:} First suppose $c\ne0$ or $d\ne0$. We claim $[\chi]\in\Sigma^m(\Fbr)$, and will use the complexes from Lemma~\ref{lem:stein_truncate}. Choose $q\ge m$ such that $c+(n-1)d \ne 0$ for any $2q\le n\le 4q$. We would like to apply Citation~\ref{cit:stab_restrict} to the cocompact action of $\Fbr$ on the $(q-1)$-connected (by Lemma~\ref{lem:stein_truncate}), hence $(m-1)$-connected space $X^{2q\le h\le 4q}$, so we need to understand how $\chi$ restricts to stabilizers. A consequence of the pure version of \cite[Corollary~2.8]{bux16} is that the stabilizer in $\Fbr$ of any cube in $X^{2q\le h\le 4q}$ is isomorphic to the stabilizer of its bottom $0$-face $x$. Hence if we can prove that for every $0$-cube $x$ in $X^{2q\le h\le 4q}$ the character $\chi$ restricts non-trivially to $\Stab_{\Fbr}(x)$ and the restriction lies in $\Sigma^m(\Stab_{\Fbr}(x))$, then Citation~\ref{cit:stab_restrict} will say that $[\chi]\in\Sigma^m(\Fbr)$. By Citation~\ref{cit:center_survives} it suffices to find an element of the center $Z(\Stab_{\Fbr}(x))$ of the stabilizer that is not killed by $\chi$. Since $\chi$ reads the same value on conjugates, and since $\Fbr$ acts transitively on $0$-cubes of $X$ with a given $h$-value, without loss of generality $x=[T,1,1]_{PB}$. The stabilizer of this $0$-cube is $PB_T=\{[T,p,T]\mid p\in PB_n\}\cong PB_n$, where $T$ has $n$ leaves, and the element $[T,\Delta_n^2,T]$ is central in this stabilizer. We have
\[
\chi([T,\Delta_n^2,T]) = c\omega_{1,n}(\Delta_n^2) + d\sum_{i=1}^{n-1}\omega_{i,i+1}(\Delta_n^2) = c+(n-1)d \ne 0 \text{,}
\]
so we are done with this case.

\textbf{Case 2:} Now assume $c=d=0$ and $a,b\ge0$, so $[\chi]$ lies in the convex hull of $[\phi_{\ty{0}}]$ and $[\phi_{\ty{1}}]$. We claim $[\chi]\in\Sigma^m(\Fbr)^c$. Since $c=d=0$, $\chi$ is induced from the split epimorphism $\Fbr \to F$, namely $\chi$ is this map composed with the character $a\chi_{\ty{0}}+b\chi_{\ty{1}}$ of $F$, where $\chi_i=\phi_i|_F$ ($i\in\{\ty{0},\ty{1}\}$). Since $m\ge 2$, the computation of $\Sigma^m(F)$ from \cite{bieri10} tells us that $[a\chi_{\ty{0}}+b\chi_{\ty{1}}]\in \Sigma^m(F)^c$. Now Citation~\ref{cit:hi_dim_quotients} implies that $[\chi]\in\Sigma^m(\Fbr)^c$.

Before considering the last case, we first need some specific results for $\Fbr(\ty{0})$ and $\Fbr(\ty{1})$, namely that $[-\phi_{\ty{0}}|_{\Fbr(\ty{0})}]\in\Sigma^m(\Fbr(\ty{0}))$ and $[-\phi_{\ty{1}}|_{\Fbr(\ty{1})}]\in\Sigma^m(\Fbr(\ty{1}))$. By Corollary~\ref{cor:Fbrn_hnn} we have $\Fbr(\ty{0})=\Fbr(\ty{01})*_{x_{\ty{0}}}$, and since $-\phi_{\ty{0}}|_{\Fbr(\ty{0})}(\Fbr(\ty{01}))=0$ and $-\phi_{\ty{0}}|_{\Fbr(\ty{0})}(x_{\ty{0}})=1$ we get  from Corollary~\ref{cor:Fbrn_Finfty} and Citation~\ref{cit:hnn_kill} that $[-\phi_{\ty{0}}|_{\Fbr(\ty{0})}]\in\Sigma^m(\Fbr(\ty{0}))$. Similarly we have $\Fbr(\ty{1})=\Fbr(\ty{10})*_{x_{\ty{1}}^{-1}}$, and since $-\phi_{\ty{1}}|_{\Fbr(\ty{1})}(\Fbr(\ty{10}))=0$ and $-\phi_{\ty{1}}|_{\Fbr(\ty{1})}(x_{\ty{1}})=1$ we get that $[-\phi_{\ty{1}}|_{\Fbr(\ty{1})}]\in\Sigma^m(\Fbr(\ty{1}))$.

\textbf{Case 3:} Now assume $c=d=0$, and $a<0$ or $b<0$. We claim $[\chi]\in\Sigma^m(\Fbr)$. Our strategy is inspired by the proof of Corollary~2.4 in \cite{bieri10}. First suppose $a<0$. By Corollary~\ref{cor:Fbrn_hnn} we have $\Fbr=\Fbr(\ty{0})*_{x_{\varnothing}^{-1}}$. By Corollary~\ref{cor:Fbrn_Finfty} and Citation~\ref{cit:hnn_restrict} it suffices to prove that $[\chi|_{\Fbr(\ty{0})}]\in\Sigma^m(\Fbr(\ty{0}))$. But $\chi|_{\Fbr(\ty{0})}=a\phi_{\ty{0}}|_{\Fbr(\ty{0})}$, so $[\chi|_{\Fbr(\ty{0})}]=[-\phi_{\ty{0}}|_{\Fbr(\ty{0})}]$ and as we just saw this indeed lies in $\Sigma^m(\Fbr(\ty{0}))$. Now suppose $b<0$, and we proceed similarly. We have $\Fbr=\Fbr(\ty{1})*_{x_{\varnothing}}$, so it suffices to prove that $[\chi|_{\Fbr(\ty{1})}]\in\Sigma^m(\Fbr(\ty{1}))$. But $\chi|_{\Fbr(\ty{1})}=b\phi_{\ty{1}}|_{\Fbr(\ty{1})}$, so $[\chi|_{\Fbr(\ty{1})}]=[-\phi_{\ty{1}}|_{\Fbr(\ty{1})}]$ and as we just saw this indeed lies in $\Sigma^m(\Fbr(\ty{1}))$.
\end{proof}

An immediate consequence of Theorem~\ref{thrm:BNSR_BF} is the following.

\begin{corollary}\label{cor:fp_Finfty}
Let $H$ be a subgroup of $\Fbr$ containing $[\Fbr,\Fbr]$. Then $H$ is finitely presented if and only if it is of type $\F_\infty$ if and only if no non-trivial character of the form $a\phi_{\ty{0}}+b\phi_{\ty{1}}$ for $a,b\ge0$ contains $H$ in its kernel.
\end{corollary}

\begin{proof}
First, Citation~\ref{cit:fin_props} says that $H$ is of type $\F_m$ if and only if no non-trivial character $\chi$ with $[\chi]\in\Sigma^m(\Fbr)^c$ contains $H$ in its kernel, so the first equivalence holds since $\Sigma^2(\Fbr)=\Sigma^\infty(\Fbr)$. The second equivalence holds by the explicit computation in Theorem~\ref{thrm:BNSR_BF}.
\end{proof}

\bibliographystyle{alpha}
\newcommand{\etalchar}[1]{$^{#1}$}

\end{document}